\documentclass[11pt]{amsart}

\usepackage[left=2.6cm,right=2.6cm,top=2.9cm,bottom=2.9cm]{geometry}

\usepackage{amsfonts}
\usepackage{color}

\usepackage{graphicx}
\usepackage{enumerate}

\newtheorem{thm}{Theorem}[section]
\newtheorem{cor}[thm]{Corollary}
\newtheorem{lema}[thm]{Lemma}
\newtheorem{prop}[thm]{Proposition}

\theoremstyle{definition}
\newtheorem{defn}[thm]{Definition}

\theoremstyle{remark}
\newtheorem{rem}[thm]{Remark}

\def\glim{\mathop{\text{\normalfont $\Gamma-$lim}}}
\def\supp{\mathop{\text{\normalfont supp}}}

\def\curl{\mathop{\text{\normalfont curl}}}
\def\diam{\mathop{\text{\normalfont diam}}}

\numberwithin{equation}{section}
\newcommand{\R}{\mathbb R}
\newcommand{\CC}{\mathbb C}
\newcommand{\Sn}{{\mathbb S}^{n-1}}

\newcommand{\N}{\mathbb N}

\newcommand{\dd}{\mathbf d}

\def\C{\mathbf {C}}
\def\J{{\mathcal{J}}}
\def\F{{\mathcal{F}}}

\newcommand{\intr}{\int_{\R^n}}

\newcommand{\ve}{\varepsilon}

\newcommand{\cd}{\rightharpoonup}

\begin{document}
\title{Magnetic Fractional order Orlicz-Sobolev spaces}

\author[J. Fern\'andez Bonder and A.M. Salort]{Juli\'an Fern\'andez Bonder and Ariel M. Salort}

\address{Departamento de Matem\'atica, FCEyN - Universidad de Buenos Aires and
\hfill\break \indent IMAS - CONICET
\hfill\break \indent Ciudad Universitaria, Pabell\'on I (1428) Av. Cantilo s/n. \hfill\break \indent Buenos Aires, Argentina.}

\email[J. Fern\'andez Bonder]{jfbonder@dm.uba.ar}
\urladdr{http://mate.dm.uba.ar/~jfbonder}

\email[A.M. Salort]{asalort@dm.uba.ar}
\urladdr{http://mate.dm.uba.ar/~asalort}




\begin{abstract}
In this paper  we define the notion of nonlocal magnetic Sobolev spaces with non-standard growth for Lipschitz magnetic fields. In this context we prove a  Bourgain - Brezis - Mironescu type formula for functions in this space as well as for sequences of functions. Finally, we deduce some consequences such as the $\Gamma-$convergence of modulars and convergence of solutions for some non-local magnetic Laplacian allowing non-standard growth laws to its local counterpart.
\end{abstract}

\maketitle
 

\section{Introduction}
The magnetic Laplacian $\Delta^A:=(\nabla -iA)^2$ plays a fundamental role in the description of particles interacting with a magnetic field  $B=\curl(A)$, where $A\colon\R^3 \to \R^3$ is the magnetic potential.

This operator can be seen as the gradient of the convex functional
$$
I^A(u) := \frac12 \int_\Omega |\nabla u - i Au|^2\, dx,
$$
in the sense that the solution $u\colon \Omega\to\CC$  to the problem
$$
-\Delta^A u = f \text{ in } \Omega,\quad u=0 \text{ on } \partial\Omega,
$$
is the unique minimizer of
$$
J^A(u) := \frac12 \int_\Omega |\nabla u - i A u|^2\, dx - \int_\Omega \Re(f \bar u)\, dx, 
$$
where for a complex number $z\in\CC$, we denote by $\Re z$ and $\Im z$ the real and imaginary parts of $z$, and $\bar z$ denotes the complex conjugate of $z$.

Several nonlinear generalizations have been studied in the past years, such as the magnetic $p-$Laplace operator ($1<p<\infty$), denoted by $-\Delta_p^A$ and defined as the gradient of the convex functional
$$
I_p^A(u) := \frac{1}{p} \int_\Omega |\nabla u - i Au|_p^p\, dx,
$$
where, for $z\in\CC^n$, $|z|_p^p := |\Re z|^p + |\Im z|^p$, $|\cdot|$ is the euclidean norm in $\R^n$ and $\Re z$, $\Im z$ denotes the real an imaginary parts of $z$ respectively. See for instance \cite{LZ} for existence results for $-\Delta_p^A$ and  \cite{BGKSM} for  this operator in the context of  graphs.

Again, the solution to
$$
-\Delta_p^A u = f \text{ in } \Omega,\quad u=0 \text{ on } \partial\Omega,
$$
is the unique minimizer of
$$
J_p^A(u) := \frac1p \int_\Omega |\nabla u - i A u|_p^p\, dx - \int_\Omega \Re(f \bar u)\, dx. 
$$

On the other hand, when studying phenomena allowing behaviors more general than power laws,  such as anisotropic fluids with flows obeying nonstandard rheology \cite{GSW,WK} or capillarity phenomena, these magnetic operators need to be extended to consider nonstandard growth different that powers or different behaviors near zero and near infinity. In these cases, Orlicz-Sobolev spaces become the natural framework to deal with. 

Given an Orlicz function $G\colon \R_+\to\R_+$ (see next section for precise definitions), and $g=G'$, the magnetic $g-$Laplace operator is defined as the gradient of the functional
\begin{equation}\label{IGA}
I_G^A(u):= \int_\Omega \left(G(|\Re(\nabla u + iAu)|) + G(|\Im(\nabla u + iAu)|)\right)\, dx
\end{equation}
and again, the solution to
$$
-\Delta_g^A u = f \text{ in } \Omega,\quad u=0 \text{ on } \partial\Omega,
$$
is the unique minimizer of
$$
J_G^A(u) := I_G^A(u) - \int_\Omega \Re(f \bar u)\, dx. 
$$

In the last decades there has been an increasing interest in the study of equations driven by nonlocal operators since they arise naturally in many important problems of nature. This fact leaded up to consider operators describing nonlocal magnetic phenomena. For instance, in the  mid 80s a fractional relativistic generalization of the magnetic Laplacian in $\R^n$ was introduced in \cite{Ic1,Ic3}, \cite[Section 3.1]{Ic2} by means of the so-called Weyl pseudo-differential operator defined with mid-point prescription
$$
\mathcal{H}_A u(x)=\frac{1}{(2\pi)^n}\int_{\R^{2n}} e^{i(x-y) \left(\xi +A\left(\tfrac{x+y}{2}\right)\right)} \sqrt{|\xi|^2+ m^2} u(y)\,dyd\xi,
$$
where here, $A\colon \R^n\to \R^n$ is a measurable function.

When $m=0$, in \cite[Eq. (3.7)]{Ic2} it is shown that for $u\in C_c^\infty(\R^n,\CC)$ the expression above can be written as
$$
\mathcal{H}_A u(x)=\frac{\Gamma\big(\tfrac{n+1}{2}\big)}{\pi^\frac{n+1}{2}}\, p.v. \int_{\R^n} \left( u(x)-e^{i(x-y) A\left(\tfrac{x+y}{2}\right)} u(y) \right) |x-y|^{-(n+1)}\,dy.
$$
Furthermore, in \cite{AvSq}, this nonlocal operator was generalized to admit a family of kernels depending on a parameter $s\in(0,1)$ as
$$
(-\Delta^A)^s u := c_{n,s}\, p.v. \int_{\R^n} \frac{u(x)-e^{i(x-y) A\left(\tfrac{x+y}{2}\right)}u(y) }{|x-y|^{n+2s}}\,dy, \qquad x\in\R^n,
$$
which recovers the expression of $\mathcal{H}_A$ for $s=\frac12$ and also recovers the fractional Laplacian when $A\equiv 0$. This operator is the gradient of the functional
$$
I_s^A(u) := \frac12 \iint_{\R^n\times\R^n} \frac{|u(x)-e^{i(x-y) A\left(\tfrac{x+y}{2}\right)}u(y)|^2}{|x-y|^{n+2s}}\, dxdy,
$$
up to some normalization constant.

The connection of this magnetic fractional laplacian $(-\Delta^A)^s$ with the classical magnetic laplacian $-\Delta^A$ was provided in \cite{SquassinaVolzone} where it is proved that their corresponding energies converge as the fractional parameter $s$ converges to 1, much in the spirit of the celebrated result of Bourgain-Brezis-Mironescu (BBM for short). See \cite{BBM}.

Recently, in \cite{NPSV} the authors introduce a fractional version of the magnetic $p-$Laplacian $-\Delta_p^A$. The magnetic fractional $p-$Laplacian considered in \cite{NPSV}, denoted by $(-\Delta_p^A)^s$, is defined as the gradient of the functional
$$
I_{s,p}^A(u) := \frac{1}{p} \iint_{\R^n\times\R^n} \frac{|u(x)-e^{i(x-y) A\left(\frac{x+y}{2}\right)} u(y)|_p^p}{|x-y|^{n+sp}}\, dxdy, \qquad x\in\R^n.
$$
Observe that for $p=2$ this definition agrees with the one given for $(-\Delta^A)^s$ and in this case, when the parameter $s$ converges to 1, one recovers the magnetic $p-$Laplace operator $-\Delta_p^A$. See \cite{NPSV} for the details.

The purpose in this work is the analysis of a fractional version of the magnetic $g-$Laplace operator $(-\Delta_g^A)^s$ and the study of the limit as the fractional parameter $s$ goes to 1.

This problem in the case of zero magnetic potential (i.e. $A=0$) was addressed in \cite{FBS}. In that paper, the authors introduced what they called the {\em fractional order Orlicz-Sobolev spaces}, as
$$
W^{s,G}(\R^n) := \left\{u\in L^G(\R^n)\colon I_{s,G}(u)<\infty \right\},
$$
for $0<s<1$, where 
$$
I_{s,G}(u) :=  \iint_{\R^n\times\R^n} G\left(\frac{ |u(x)-u(y)| }{|x-y|^s}\right) \frac{dxdy}{|x-y|^n}
$$
and
$$
L^G(\R^n) := \left\{u\in L^1_{\text{loc}}(\R^n)\colon \int_{\R^n} G(|u|) \, dx <\infty\right\}.
$$

Then, in \cite{FBS}, they went on to define the fractional $g-$Laplace operator $(-\Delta_g)^s$ as the gradient of the functional $I_{s,G}$ and prove the convergence of this fractional operator to the (by now) classical $g-$Laplace operator $-\Delta_g$.

To this end, we consider a Lipschitz magnetic potential $A\colon \R^n\to \R^n$ and an Orlicz function $G$. Then, the \emph{magnetic fractional $g-$Laplace operator} $(-\Delta_g^A)^s$  is defined as the gradient of the non-local energy functional
\begin{equation}\label{IsGA}
I_{s,G}^A(u) := \iint_{\R^n\times\R^n} \left(G\left(|\Re(D_s^A u(x,y))|\right) + G\left(|\Im(D_s^A u(x,y))|\right)\right)\, \frac{dxdy}{|x-y|^n},
\end{equation}
where $D_s^A u(x,y)$ is the magnetic H\"older quotient of order $s$ defined as
\begin{equation}\label{DsA}
D_s^A u(x,y) := \frac{u(x)-e^{i(x-y) A\left(\frac{x+y}{2}\right)} u(y)}{|x-y|^s}.
\end{equation}
Observe that when $G(t) = \frac{1}{p} t^p$ we recover the functional $I_{s,p}^A$.

In this manuscript we will be interested in the behavior of   $I_{s,G}^A$ as $s\uparrow 1$ and its connection with the local energy functional $I_G^A$ which is closely related with  the  magnetic $g-$Laplace operator $-\Delta_g^A$ as we mentioned above.

Our first result states a magnetic Bourgain-Brezis-Mironescu identity for fractional Orlicz-Sobolev functions.

To this end, following \cite{FBS}, given an Orlicz function $G$, we define its {\em spherical limit} as
\begin{equation} \label{phitilde} \tag{S}
\tilde G(a):=\lim_{s\uparrow 1} (1-s)\int_0^1  \int_{\Sn}   G\left( a |z_N| r^{1-s}\right)   dS_z \frac{dr}{r}.
\end{equation}
provided that this limit exists.

\begin{thm} \label{main1}
Let $1<p_-\le p_+<\infty$ be fixed and let $G$ be an Orlicz function satisfying the growth condition
\begin{equation} \label{cond.intro} \tag{L}
p_-\le \frac{tG'(t)}{G(t)}\le p_+, 
\end{equation}
and such that the limit in \eqref{phitilde} exists. Let $A\colon\R^n\to \R^n$ be a Lipschitz continuous function. Then, for any $u\in L^G(\R^n;\CC)$, it holds that
\begin{equation} \label{bbm}
\lim_{s\uparrow 1} (1-s) I_{s,G}^A(u) = I_{\tilde G}^A(u),
\end{equation}
where $\tilde G$ is defined by \eqref{phitilde}.
\end{thm}

\begin{rem}\label{rema}
In view of \cite[Proposition 2.16]{FBS}, whenever $\tilde G\colon \R^+\to \R^+$ is well defined, there exist positive constant $c_1$ and $c_2$ such that
$$
c_1 G(t) \leq  \tilde G(t) \leq  c_2  G(t)\quad \text{for every } t>0.
$$
We refer to \cite{FBS} for the explicit computation of $\tilde G$ in some particular examples.
\end{rem}

\begin{rem}
The limit in \eqref{bbm} is understood in the sense that if $u\in W^{1,G}_A(\R^n)$ then the limit is finite and coincides with $I_{\tilde G}^A(u)$ and, if 
$$
\liminf_{s\uparrow 1} I_{s,G}^A(u) < \infty,
$$
then $u\in W^{1,G}_A(\R^n)$ and \eqref{bbm} holds.
\end{rem}

As a consequence of Theorem \ref{main1} we deduce some $\Gamma-$convergence results for the modulars and therefore the convergence of the solutions of the magnetic fractional $g-$Laplace operator to its the local magnetic counterpart. In fact, our result on the convergence of the operators $(-\Delta_g^A)^s$ to $-\Delta_g^A$ reads as follows.
\begin{thm}\label{main3}
Let $G$ be an Orlicz function satisfying \eqref{cond.intro} such that \eqref{phitilde} exists and let $A\colon \R^n\to\R^n$ be a Lipschitz continuous function. Let $G^*$ be the Legendre's transform of $G$, $\Omega\subset \R^n$ a bounded open set and $f\in L^{G^*}(\Omega; \CC)$.

For each $0<s<1$, let $u_s\in L^G(\Omega;\CC)$ be the unique solution to
$$
(-\Delta_g^A)^s u_s = f\ \text{ in }\Omega,\qquad u_s=0\ \text{ in }\R^n\setminus \Omega.
$$
Then $u_s\to u$ as $s\to 1$ in $L^G(\Omega;\CC)$ where $u$ is the unique solution to
$$
-\Delta_{\tilde g}^A u = f\ \text{ in }\Omega,\qquad u=0\ \text{ in } \partial \Omega,
$$
where $\tilde g = \tilde G'$.
\end{thm}

We observe that this last result seems to be new, even in the magnetic $p-$laplacian setting.

\subsection*{Organization of the paper}
In section 2, we collect some preliminaries on Orlicz functions that will be used throughout the paper, define the magnetic Orlicz-Sobolev spaces and prove some elementary properties of these spaces.

In section 3 we prove some technical results needed in the proof of our main results.

Section 4 is devoted to the proof of Theorem \ref{main1}.

Finally, in section 5, we derive some consequences of Theorem \ref{main1} and, in particular, we show the proof of Theorem \ref{main3}.

\section{Preliminaries} \label{preliminar}
\subsection{Orlicz functions}

We start by recalling the definition of the well-known Orlicz functions.
\begin{defn}\label{defi.Orlicz}
$G\colon \R_+ \to \R_+$ is called an \emph{Orlicz function} if it can be written as
$$
G(x)=\int_0^x g(t)\,dt
$$
where the real-valued function $g$ defined on $\R_+$ is positive, right continuous, nondecreasing and $g(0)=0$ and $g(x)\to \infty$ as $x\to\infty$.
\end{defn}

It is easy to see that an Orlicz function $G$ satisfies the following properties.
\begin{align}
\tag{$P_1$}\label{H1} &G \text{ is Lipschitz continuous, convex,  increasing and }  G(0)=0.\\ 
\tag{$P_2$}\label{P5} & G(ab)\leq bG(a) \text{ for any } 0<b<1, a>0.\\ 
\tag{$P_3$}\label{H3} &G \text{ is super-linear at zero, that is } \lim_{x\downarrow 0} \frac{G(x)}{x} = 0.
\end{align}

We say that an Orlicz function $G$ satisfies the \emph{$\Delta_2$ condition} if there exists $\C>2$ such that 
\begin{equation} \label{H2} \tag{$\Delta_2$}
G(2x)\leq \C G(x) \quad \text{ for all } x\in \R_+.
\end{equation}

From \eqref{H2} it is easy to see that for any $a,b\geq 0$
\begin{equation}  \label{P4}
G(a+b)\leq \tfrac{\C}{2} (G(a)+G(b))
\end{equation}
where $\C$ is the constant in the $\Delta_2$ condition.

In \cite[Theorem 4.1]{KrRu61} it is shown that the $\Delta_2$ condition is equivalent to
$$
\frac{tg(t)}{G(t)}\le p_+
$$
for some $p_+>1$ (then the constant in \eqref{H2} is just $2^{p_+}$).

For most of our computations we will require the stronger hypothesis
\begin{equation}\label{lieberman} \tag{L}
1<p_-\le \frac{tg(t)}{G(t)}\le p_+<\infty.
\end{equation}
The lower inequality in \eqref{lieberman} is easily seen as being equivalent to the $\Delta_2$ condition of the \emph{complementary function} (or {\em Legendre's transform}) of $G$, which is defined as
$$
G^*(s) := \sup_{t>0}\{st-G(t)\}.
$$
Therefore, condition \eqref{lieberman} is equivalent to the fact that both $G$ and $G^*$ satisfy the \ref{H2} condition. Let us recall that this is what is needed in order for the Orlicz space $L^G(\R^n;\CC)$ to be reflexive. See \cite{KrRu61} and the next subsection for definitions and properties of Orlicz spaces.

Moreover, it is easy to check that \eqref{lieberman} implies that
\begin{equation} \label{cotas}
\min\{a^{p_+}, a^{p_-}\} G(b) \leq G(ab)\leq \max\{a^{p_+}, a^{p_-}\} G(b), \qquad a,b\in \R_+.
\end{equation}

\subsection{Magnetic Fractional Orlicz--Sobolev spaces}
Given an Orlicz function $G$, a fractional parameter $0<s< 1$ and a function $A\colon \R^n\to\R^n$,  we consider the spaces $L^G(\R^n;\CC)$ and $W^{s,G}_A(\R^n)$ defined as 
\begin{align*}
&L^G(\R^n;\CC) :=\left\{ u\colon \R^n \to \CC \text{ measurable}\colon  I_{G}(u) < \infty \right\},\\
&W^{s,G}_A(\R^n):=\left\{ u\in L^G(\R^n;\CC)\colon I_{s,G}^A(u)<\infty \right\},
\end{align*}
where the modulars $I_G$ is defined as
$$
I_{G}(u) := \int_{\R^n} G(|\Re u|) + G(|\Im u|)\,dx,
$$
and $I_{s,G}^A$ is defined \eqref{IsGA}.

We also define the space $W^{1,G}_A(\R^n)$ as
$$
W^{1,G}_A(\R^n):= \{u\in W^{1,1}_\text{loc}(\R^n;\CC)\colon I_G(u), I_G^A(u)<\infty\},
$$
where $I_G^A(u)$ is defined in \eqref{IGA}.

Along this paper we will always assume that $A$ is a bounded and Lipschitz continuous function.

In these spaces we consider the Luxemburg norm defined through the modulars $I_{s, G}^A$, namely
$$
\|u\|_{s, G}^A = \|u\|_G + |u|_{s,G}^A, 
$$
where 
$$
\|u\|_G := \inf\{\lambda>0\colon I_G(\tfrac{u}{\lambda})\le 1\}
$$
is the usual (Luxemburg) norm on $L^G(\R^n;\CC)$ and
$$
|u|_{s,G}^A := \inf\{\lambda>0\colon I_{s,G}^A(\tfrac{u}{\lambda})\le 1\}.
$$

\begin{rem}\label{equiv.I}
Observe that if $z\in\CC$, then
$$
G(|\Re z|) + G(|\Im z|) \le 2 G(|z|) \qquad\text{and}\qquad G(|z|)\le \C (G(|\Re z|) + G(|\Im z|)),
$$
where $\C$ is the constant in the $\Delta_2$ condition. Hence the functionals $I_{s,G}^A$ and $I_G$ turn out to be equivalent to 
$$
\tilde I_{s,G}^A(u) := \iint_{\R^n\times\R^n} G(|D_s^A u(x,y)|)\, \frac{dxdy}{|x-y|^n}\quad \text{and}\quad \tilde I_G(u) := \intr G(|u|)\, dx
$$
respectively.
\end{rem}

\section{Some technical results}
In this section we establish some properties on magnetic Orlicz-Sobolev the spaces and prove some useful properties on magnetic modulars. Finally we state a compactness result in $W^{s,G}_A(\R^n)$.

\begin{prop} \label{es.denso}
$C^\infty_c(\R^n;\CC)$ is dense in $W^{1,G}_A(\R^n)$ provided that the Orlicz function $G$ satisfies the $\Delta_2$ condition.
\end{prop}

\begin{proof}
The proof is completely analogous to that of \cite[Theorem 7.22]{LossLieb} with the obvious modifications and using the $\Delta_2$ condition.
\end{proof}

\begin{prop}\label{propiedades1}
Let $G$ be an Orlicz function satisfying the $\Delta_2$ condition. Then the spaces $L^G(\R^n;\CC)$ and $W_A^{1,G}(\R^n)$ are separable Banach spaces. 

If we further assume \eqref{lieberman},  then the dual space of $L^G(\R^n;\CC)$ can be identified with $L^{G^*}(\R^n;\CC)$. Moreover, $L^G(\R^n;\CC)$ and $W^{1, G}_A(\R^n)$ are reflexive spaces.
\end{prop}

\begin{proof}
The proof is standard and it is omitted.
\end{proof}

\subsection{Modular of convolutions}\label{convolutions}
In this paragraph we analyze the behavior of the modular of convolutions. As usual, we denote by $\rho\in C^\infty_c(\R^n)$ the standard mollifier with $\supp(\rho)=B_1(0)$ and $\rho_\ve(x)=\ve^{-n}\rho(\tfrac{x}{\ve})$ is the approximation of the identity. It follows that $\{\rho_\ve\}_{\ve>0}$ is a family of positive functions satisfying 
$$
\rho_\ve\in C_c^\infty(\R^n), \quad \supp(\rho_\ve)=B_\ve(0), \quad \intr \rho_\ve \, dx =1.
$$
Given $u\in L^G(\R^n;\CC)$ we define the regularized functions $u_\ve\in L^G(\R^n;\CC)\cap C^\infty(\R^n;\CC)$ as
\begin{equation} \label{regularizada}
u_\ve(x)=u*\rho_\ve(x).
\end{equation}

In this context we prove the following useful estimate on regularized functions.
\begin{lema} \label{lema.reg}
Given an Orlicz function $G$ satisfying the $\Delta_2$ condition, let $u\in L^G(\R^n;\CC)$ and $\{u_\ve\}_{\ve>0}$ be the family defined in \eqref{regularizada}. Then there exists a constant $C$ depending on $n$, $\|A\|_\infty$ and $\C$, the constant in \eqref{H2}, such that
$$
I_{s,G}^A(u_\ve)  \leq  C \left( I_{s,G}^A(u) +  \left(\frac{1 }{s}+\frac{1}{1-s}\right)I_G(u)\right),
$$
for all $\ve>0$ and $0<s<1$.
\end{lema}

\begin{proof}
By Remark \ref{equiv.I}, it is enough to prove the result for the functionals $\tilde I_{s,G}^A$ and $\tilde I_G$.

First, observe that the modular $\tilde I_{s,G}^A(u_\ve)$ can be expressed as
\begin{equation}\label{ec.g.eps}
\tilde I_{s,G}^A(u_\ve) = \iint_{\R^n\times\R^n} G\left(|D_s^A u_\ve(x,x+h)|\right) \, \frac{dxdh }{|h|^n}.
\end{equation}
Now, observe that
\begin{align*}
|D_s^A u_\ve(x,x+h)| \le & \intr \left|\frac{u(x-y) - e^{-ihA(x+\frac{h}{2})}u(x-y+h)}{|h|^2}\right| \rho_\ve(y)\, dy\\
\le & \intr |D_s^A u(x-y,x-y+h)| \rho_\ve(y)\, dy \\
& + \intr \frac{|e^{-ihA(x+\frac{h}{2})} - e^{-ihA(x-y+\frac{h}{2})}| |u(x-y+h)|}{|h|^s} \rho_\ve(y)\, dy.
\end{align*}
Next, we use that $|1-e^{it}|\le |t|$ for $|h|<1$ and $|1-e^{it}|\le 2$ for $|h|\ge 1$, and we obtain the bound
$$
|D_s^A u_\ve(x,x+h)|\le \intr |D_s^A u(x-y,x-y+h)| \rho_\ve(y)\, dy + C \min\{|h|^s, |h|^{1-s}\} \intr |u(x-y+h)| \rho_\ve(y)\, dy,
$$
where $C$ depends on $\|A\|_\infty$.

Now, using \eqref{H2} and Jensen's inequality, we get
\begin{align*}
 G\big( |D_s^A u_\ve(x,x+h)| \big)\le & \C \intr G(|D_s^A u(x-y,x-y+h)|) \rho_\ve(y)\, dy\\
 & + \C \intr G(C \min\{|h|^s, |h|^{1-s}\} |u(x-y+h)|)\rho_\ve(y)\, dy\\
 = & \C ((i) + (ii)).
\end{align*}

Integrating $(i)$ over $\R^n\times\R^n$, using Fubini's theorem and the fact that $\intr \rho_\ve\, dy = 1$, we find that
\begin{align} \label{ec.g.eps.1}
\iint_{\R^n\times\R^n}\intr G(|D_s^A u(x-y,x-y+h)|) \rho_\ve(y)\, dy\, \frac{dxdh}{|h|^n} = \tilde I_{s,G}^A(u).
\end{align}
Now, we deal with the integral of $(ii)$. First we observe that from \eqref{H2} it follows that  $G(Ct)\le \C^\kappa G(t)$ where $\kappa\in\N$ is such that $2^{\kappa-1}\le C < 2^\kappa$. Hence, integrate $(ii)$ over $\R^n\times \R^n$ and obtain
\begin{align*}
\iint_{\R^n\times\R^n} \intr &G(C \min\{|h|^s, |h|^{1-s}\} |u(x-y+h)|)\rho_\ve(y)\, dy\, dx\, \frac{dh}{|h|^n} \\
\le & \C^\kappa \iint_{\R^n\times\R^n} \intr G(\min\{|h|^s, |h|^{1-s}\} |u(x-y+h)|)\rho_\ve(y)\, dy\, dx\, \frac{dh}{|h|^n}\\
\le &\C^\kappa  \int_{|h|<1} \intr \intr G(|u(x-y+h)|)\rho_\ve(y)\, dy\, dx\, \frac{dh}{|h|^{n-1+s}} \\
& + \C^\kappa  \int_{|h|\ge 1} \intr \intr G(|u(x-y+h)|)\rho_\ve(y)\, dy\, dx\, \frac{dh}{|h|^{n-s}}.
\end{align*}

Next, we use Fubini's theorem and the fact that $\intr\rho_\ve\, dy=1$ to find that
\begin{align*}
\iint_{\R^n\times\R^n} \intr &G(C \min\{|h|^s, |h|^{1-s}\} |u(x-y+h)|)\rho_\ve(y)\, dy\, dx\, \frac{dh}{|h|^n}\le  C\left(\frac{1}{s} + \frac{1}{1-s}\right) \tilde I_G(u).
\end{align*}
The proof is now complete.
\end{proof}

\subsection{Modular of truncations}
Let us estimate the behavior of modulars of truncated functions.  Let $\eta\in C_c^\infty(\R^n)$ such that $\eta=1$ in $B_1(0)$, $\supp (\eta)=B_2(0)$, $0\leq \eta\leq 1$ in $\R^n$ and $\|\nabla \eta\|_\infty\le 2$. Given $k\in\N$ we define $\eta_k(x)=\eta(\tfrac{x}{k})$. Observe that  $\{\eta_k\}_{k\in\N} \in C_c^\infty(\R^n)$ and
$$
0\leq \eta_k \leq 1, \quad \eta_k =1 \text{ in } B_k(0), \quad  \supp (\eta_k)=B_{2k} (0),\quad  |\nabla \eta_k|\leq \frac{2}{k}.
$$
Given $u\in L^G(\R^n;\CC)$ we define the truncated functions $u_k$, $k\in\N$ as 
\begin{equation} \label{truncada}
u_k=\eta_k u.
\end{equation}

In the next lemma we analyze the behavior of the modular of truncated functions. 

\begin{lema} \label{lema.trunc}
Given an Orlicz function $G$ satisfying \eqref{H2}, let $u\in L^G(\R^n;\CC)$ and $\{u_k\}_{k\in\N}$ be the functions defined in \eqref{truncada}. Then there exists a constant $C$ depending on $n$, $\|A\|_\infty$ and $\C$, the constant in the $\Delta_2$ condition, such that
$$
I_{s,G}^A(u_k) \leq C\left( I_{s,G}^A(u) +  \left(\frac{1}{s} + \frac{1}{k(1-s)}\right)I_G (u)\right).
$$
\end{lema}

\begin{proof}
As in the previous proof, by Remark \ref{equiv.I} is enough to prove the Lemma for the functionals $\tilde I_{s,G}^A$ and $\tilde I_G$.

Observe first that $D_s^A u_k(x,y)=\eta_k(y) D_s^A u(x,y) + u(x) D_s\eta_k(x,y)$, where 
$$
D_s\eta_k(x,y) = \frac{\eta_k(x)-\eta_k(y)}{|x-y|^s}.
$$
Then, from \eqref{H2} and since $\eta_k\leq 1$ we have
$$
G(|D_s^A u_k(x,y)|) \leq \frac{\C}{2}  G(|D_s^A u(x,y)|) + \frac{\C}{2}  G(|u(x)| |D_s \eta_k(x,y)|).
$$
Then we get
\begin{align*}
\tilde I_{s,G}^A(u_k)  \leq  \frac{\C}{2} \tilde I_{s,G}^A(u) + \frac{\C}{2}  \iint_{\R^n\times\R^n}  G(|u(x)| |D_s \eta_k(x,y)|)\,  \frac{dxdy}{|x-y|^n}.
\end{align*}

The integral above can be splitted as follows
\begin{align*}
\left(\intr \int_{|x-y|\geq 1} + \intr \int_{|x-y|< 1}\right)  G(|u(x)| |D_s \eta_k(x,y)|)\, \frac{dx dy}{|x-y|^n} := I_1 + I_2.
\end{align*}

The monotonicity of $G$ and \eqref{cotas} allow us to bound $I_1$ as follows
\begin{align*}
I_1 &\leq \intr\int_{|x-y|\geq 1}  \frac{G (2 |u(x)|)}{|x-y|^{n+s}}\,dxdy\le \C n\omega_n \int_1^\infty \frac{1}{r^{s+1}}dr \intr G(|u(x)|) \,dx= \frac{\C n\omega_n}{s} \tilde I_G(u) .
\end{align*}

We deal now with $I_2$. Observe that, since $|\nabla \eta_k|\leq \tfrac{2}{k}$ and \eqref{cotas} holds, 
\begin{align*}
I_2 &\leq \intr \int_{|x-y|\leq 1} G\left(\frac{2}{k} \frac{ |u(x)| }{  |x-y|^{s-1}}\right) \frac{dx dy}{|x-y|^n} \leq \frac{n\omega_n \C}{k}  \intr \int_0^1 G ( |u(x)|  )  \, \frac{dr}{r^s}dx = \frac{n\omega_n \C}{k(1-s)} \tilde I_G(u),
\end{align*}
where we have used \eqref{H2} in the last inequality.

From these estimates the conclusion of the lemma follows.
\end{proof}

\subsection{A compactness result for $W^{s,G}_A(\R^n)$ spaces.}
In this subsection we prove the compactness of the immersion $W^{s,G}_A$ into $L^G$. The proof lies   on a variant of the well-known Fr\`echet-Kolmogorov Compactness Theorem.

\begin{thm} \label{teo.comp}
Let $0<s<1$ and $G$ an Orlicz function satisfying \eqref{H2}. Then for every bounded sequence $\{u_n\}_{n\in\N}\subset W^{s,G}_A(\R^n)$, i.e., $\sup_{n\in\N} ( I_{s,G}^A(u_n) + I_G^A(u_n) )<\infty$, there exists $u\in W^{s,G}_A(\R^n)$ and a subsequence $\{u_{n_k}\}_{k\in\N}\subset \{u_n\}_{n\in\N}$ such that $u_{n_k}\to u$ in $L^G_{\text{loc}}(\R^n;\CC)$.
\end{thm}

This theorem is an immediate consequence of the analogous compactness result for the inclusion $W^{s,G}(\R^n)\subset L^G_{\text{loc}}(\R^n;\CC)$ proven in \cite[Theorem 3.1]{FBS} combined with the next result.
\begin{lema}\label{lema.A.0}
Let $G\colon \R_+\to\R_+$ be an Orlicz function verifying the $\Delta_2$ condition and let $A\colon \R^n\to\R^n$ be a bounded magnetic potential. Then
$$
W^{s,G}(\R^n) = W^{s,G}_A(\R^n).
$$
Moreover, there exists $C>0$ depending on $n$, $s$, $\|A\|_\infty$ and $\C$ such that
$$
I_{s,G}^A(u)\le C \left(I_{s,G}(u) + \left(\frac{1}{s} + \frac{1}{1-s}\right)I_G(u)\right),
$$
$$
I_{s,G}(u)\le C \left(I_{s,G}^A(u) + \left(\frac{1}{s} + \frac{1}{1-s}\right)I_G(u)\right).
$$
\end{lema}

\begin{proof}
By Remark \ref{equiv.I} is enough to prove the lemma for the functionals $\tilde I_{s,G}^A$, $\tilde I_G$ and $\tilde I_{s,G}$, where
$$
\tilde I_{s, G}(u) := \iint_{\R^n\times\R^n} G(|D_s u(x,y)|)\, \frac{dxdy}{|x-y|^n},
$$ 
and $D_s u(x,y) := \frac{u(x)-u(y)}{|x-y|^n} = D_s^0 u(x,y)$.

Assume first that $u\in W^{s,G}(\R^n)$. Then
$$
|D_s^A u(x,y)| \le |D_s u(x,y)| + \frac{|1-e^{i(x-y)A(\frac{x+y}{2})}|}{|x-y|^s} |u(y)|.
$$
Therefore, from \eqref{H2} we obtain
$$
G(|D_s^A u(x,y)|) \le \C \left( G(|D_s u(x,y)|) + G\left(\frac{|1-e^{i(x-y)A(\frac{x+y}{2})}|}{|x-y|^s} |u(y)|\right)\right)
$$
and so
$$
\tilde I^A_{s,G}(u)\le \C\left(\tilde I_{s,G}(u) + \iint_{\R^n\times\R^n} G\left(\frac{|1-e^{i(x-y)A(\frac{x+y}{2})}|}{|x-y|^s} |u(y)|\right)\, \frac{dxdy}{|x-y|^n}\right).
$$

Now, using that 
$$
|1-e^{i(x-y)A(\frac{x+y}{2})}| \le \begin{cases}
2 & \text{if } |x-y|\ge 1\\
\|A\|_\infty |x-y| & \text{if } |x-y|<1,
\end{cases}
$$
the last integral is bounded as
\begin{align*}
&\iint_{\R^n\times\R^n} G(|1-e^{i(x-y)A(\frac{x+y}{2})}| |u(y)|)\, \frac{dxdy}{|x-y|^n} \le\\
&\intr \left(\int_{|x-y|<1} G(\|A\|_\infty |x-y|^{1-s} |u(y)|)\, \frac{dx}{|x-y|^n} + \int_{|x-y|\ge 1} G(2|x-y|^{-s}|u(y)|)\, \frac{dx}{|x-y|^n}\right)\, dy\\
&\le C\intr G(|u(y)|) \left(\int_{|x-y|<1} \frac{dx}{|x-y|^{n+s-1}} + \int_{|x-y|\ge 1} \frac{dx}{|x-y|^{n+s}}\right)\, dy\\
&= C\left(\frac{1}{s} + \frac{1}{1-s}\right) \tilde I_G(u).
\end{align*}

So we arrive at 
$$
\tilde I_{s,G}^A(u)\le C\left(\tilde I_{s,G}(u) + \left(\frac{1}{s} + \frac{1}{1-s}\right)\tilde I_G(u)\right).
$$

On the other hand, if $u\in W^{s,G}_A(\R^n)$, 
$$
|D_s u(x,y)|\le |D_s^A u(x,y)| + \frac{|1-e^{i(x-y)A(\frac{x+y}{2})}|}{|x-y|^s} |u(y)|
$$
and arguing exactly as before, we obtain
$$
\tilde I_{s,G}(u)\le C\left(\tilde I_{s,G}^A(u) + \left(\frac{1}{s} + \frac{1}{1-s}\right)\tilde I_G(u)\right).
$$
The proof is complete.
\end{proof}

\section{A BBM formula in $W^{s,G}_A(\R^n)$}
In this section we prove our first main results. Our proof makes use of the following two key lemmas.

\begin{lema} \label{teo1}
Let $G$ be an Orlicz function satisfying \eqref{H2} and let $A$ be a Lipschitz magnetic field. Then there exists a constant $C$ depending on $n$, $\|A\|_\infty$, $\|\nabla A\|_\infty$ and $\C$, the constant in the $\Delta_2$ condition, such that
$$
I_{s,G}^A(u) \leq C\left( \left(\frac{1}{1-s} + \frac{1}{s}\right) I_G(u) + \frac{1}{1-s} I_G^A(u)\right).
$$
\end{lema}

\begin{proof}
Once again, by Remark \ref{equiv.I} it is equivalent to prove the result for the functionals $\tilde I_{s,G}^A$, $\tilde I_G$ and $\tilde I_G^A$.

Let us first assume that $u\in C^1_c(\R^n;\CC)$ and split $\tilde I_{s,G}^A(u)$ as follows
$$
\iint_{\R^n\times\R^n} G(|D_s^A u(x,y)|) \frac{dx\,dy}{|x-y|^n} :=I_1+I_2,
$$
where $I_1$ denotes the integral over $|x-y|<1$ and $I_2$ over its complement.

Let us bound $I_1$.  For a fixed $x\in\R^n$, let us denote for the moment $\phi(y) = e^{i(x-y)A(\frac{x+y}{2})}u(y)$. Therefore we can write
$$
\phi(x) - \phi(y) = \int_0^1 \tfrac{d}{dt} \phi(tx+(1-t)y) \,dt= \int_0^1 \nabla \phi(tx+(1-t)y)\cdot (x-y) \,dt.
$$

A direct computation gives that for a.e. $x,y\in \R^n$
\begin{align*}
\nabla \phi(y) &=e^{i(x-y)A\left( \frac{x+y}{2}\right)} \nabla u(y)- i \left(A\left(\tfrac{x+y}{2}\right) + \tfrac{y-x}{2}\nabla A\left(\tfrac{x+y}{2}\right)  \right) e^{i(x-y) A\left( \frac{x+y}{2}\right)} u(y)
\end{align*}
from where,   
\begin{align*}
|\nabla \phi(y)|  &\leq |\nabla u(y)-iA(y)u(y)|+ (|A(\tfrac{x+y}{2})-A(y)|+\tfrac12 \|\nabla A\|_\infty |x-y|)|u(y)|\\
&\leq |\nabla u(y)-iA(y)u(y)|+\|\nabla A\|_\infty |x-y||u(y)|.
\end{align*}
Since $|x-y|<1$, we get
\begin{align*}
|D_s^A u(x,y)| \leq& \int_0^1 |\nabla u(tx+(1-t)y)-iA(tx+(1-t)y) u(tx+(1-t)y)||x-y|^{1-s}\,dt\\
&+ \|\nabla A\|_\infty \int_0^1 |u(tx+(1-t)y)||x-y|^{1-s}\,dt.
\end{align*}
Now, by using Jensen's inequality and \eqref{H2}
\begin{align*}
G(|D_s^A u(x,y)|) \leq C&\left( \int_0^1 G\left(|\nabla u(tx+(1-t)y)-iA(tx+(1-t)y) u(tx+(1-t)y)||x-y|^{1-s} \right)\,dt\right. \\
&+ \left. \int_0^1 G\left( |u(tx+(1-t)y)||x-y|^{1-s} \right)\,dt\right),
\end{align*}
where $C$ depends on $\C$ and $\|\nabla A\|_\infty$.

Then, since $|x-y|<1$, from \eqref{P5}
\begin{align*}
I_1\leq& C\left(\iint_{|x-y|<1} \int_0^1 G\left(|\nabla u(tx+(1-t)y)-iA(tx+(1-t)y) u(tx+(1-t)y)|  \right)\,dt\frac{dxdy}{|x-y|^{n-1+s}}\right.\\
&+ \left. \iint_{|x-y|<1}  \int_0^1  G\left(|u(tx+(1-t)y)|   \right)\,dt \frac{dxdy}{|x-y|^{n-1+s}}\right) \\
\leq& C\left(\int_{|z|<1} |z|^{1-s-n} \,dz \right) \left( \int_{\R^n} G\left(|\nabla u -iA(x) u | \right)\,dx + \int_{\R^n}   G(|u|)\,dx \right).
\end{align*}
Finally, by using polar coordinates we get
\begin{equation} \label{cota.I1}
I_1 \leq C\frac{1}{1-s} \left( \tilde I_{1,G}^A(u) + \tilde I_G(u) \right),
\end{equation}
with $C$ depending on $n$, $\C$ and $\|\nabla A\|_\infty$.

The term $I_2$ can be bounded using \eqref{cotas} and \eqref{H2}. Indeed, 
\begin{align*}
I_2&\leq \iint_{|x-y|\geq 1} G(|u(x)-e^{i(x-y)A(\frac{x+y}{2})}u(y)|)\, \frac{dxdy}{|x-y|^{n+s}}\\
&\leq \C \iint_{|x-y|\geq 1} (G(|u(x))|)+G(|u(y)|))\,\frac{dxdy}{|x-y|^{n+s}}\\
&= \C \int_{|h|\ge 1} \intr (G(|u(x)|) + G(|u(x-h)|))\,dx\frac{dh}{|h|^{n+s}}\\
&= 2\C \int_1^\infty r^{-s-1}\,dr \int_{\R^n} G(|u(x)|)\,dx\\
&= \frac{2\C n \omega_n}{s} \tilde I_G(u).
\end{align*}

This concludes the proof of the lemma for $u\in C^1_c(\R^n)$.

Finally, by Lemma \ref{es.denso}, given $u\in W^{1,G}_A(\R^n)$ one can take a sequence $\{u_k\}_{k\in\N}\subset C^1_c(\R^n;\CC)$ such that $u_k\to u$ in $W^{1,G}_A(\R^n)$ and without loss of generality, we may assume that $u_k\to u$ a.e. in $\R^n$. It implies that
$$
G(|D_s^A u_k(x,y)|)\to G(|D_s^A u(x,y)|)\quad \text{a.e. in } \R^n\times\R^n.
$$
Therefore, by Fatou's Lemma, we obtain that
\begin{align*}
\tilde I_{s,G}^A(u)&\le \liminf_{k\to\infty} \tilde I_{s,G}^A(u_k) \le C  \lim_{k\to\infty}\left( \left( \frac{1}{s} + \frac{1}{1-s}\right) \tilde I_G(u_k) +   \frac{1}{1-s} \tilde I_G^A(u_k) \right)\\
&= C  \left( \left( \frac{1}{s} + \frac{1}{1-s}\right) \tilde I_G(u) +   \frac{1}{1-s} \tilde I_G^A(u) \right).
\end{align*} 
The proof is now complete.
\end{proof}

\begin{lema} \label{lemma.uno}
Let $G$ be an Orlicz function satisfying \eqref{H2} such that the limit in \eqref{phitilde} exists and $u\in C^2_c(\R^n;\CC)$. Then, for every fixed $x\in \R^n$ we have that
\begin{equation} \label{eq.bbm}
\lim_{s\uparrow 1} (1-s) \int_{\R^n} G(|\Re D_s^A u(x,y)|) \frac{dy}{|x-y|^n} =  \tilde G(|\Re(\nabla u(x)-iA(x)u(x))|)
\end{equation}
and
\begin{equation} \label{eq.bbm2}
\lim_{s\uparrow 1} (1-s) \int_{\R^n} G(|\Im D_s^A u(x,y)|) \frac{dy}{|x-y|^n} = \tilde G(|\Im(\nabla u(x)-iA(x)u(x))|),
\end{equation}
where $\tilde G$ is defined in \eqref{phitilde}.
\end{lema}

\begin{proof}
Let us prove \eqref{eq.bbm}. The formula \eqref{eq.bbm2} follows analogously.

For each fixed $x\in \R^n$ we split the integral 
\begin{align*}
\int_{\R^n} G(|\Re D_s^A u(x,y)|)\, \frac{dx}{|x-y|^n}:=  I_1 + I_2,
\end{align*}
where $I_1$ denotes the integral over the set $\{y\in \R^n: |x-y|<1\}$, and $I_2$ over its complement. 

For each fixed $x\in\R^n$, let $\phi(y) = e^{i(x-y)A(\frac{x+y}{2})}u(y)$. Since $u\in C^2_c(\R^n;\CC)$, we have that $\phi\in C^2_c(\R^n;\CC)$ and hence we have
\begin{equation}\label{taylor.phi}
\phi(y) = \phi(x) + \nabla \phi(x) (y-x) + O(|x-y|^2),
\end{equation}
where the big-$O$ depends on the $C^2$ norm of $u$, on $\|A\|_\infty$ and on $\|\nabla A\|_\infty$.

Observe that 
\begin{equation}\label{nabla.phi}
\phi(x) = u(x) \quad \text{and}\quad \nabla\phi(x) = (\nabla u(x) - iA(x)u(x))(x-y).
\end{equation}

Combining \eqref{taylor.phi} and \eqref{nabla.phi} we arrive at
\begin{align*}
D_s^A u(x,y) = (\nabla u(x) - iA(x)u(x))\frac{(x-y)}{|x-y|^s} + O(|x-y|^{2-s}).
\end{align*}

Hence, since $G$ is Lipschitz continuous, for any $x,y\in \R^n$, $x\neq y$ we have that
\begin{align*}
&\left| G(|\Re D_s^A u(x,y)|) - G\left(\left |\Re\left((\nabla u(x)-iA(x)u(x)) \frac{x-y}{|x-y|^s}\right)\right|\right) \right|\\
&\quad \leq C \left|D_s^A u(x,y) - (\nabla u(x)-iA(x)u(x)) \frac{(x-y)}{|x-y|^s}\right|\\
&\quad \leq  C |x-y|^{2-s}.
\end{align*}

From this estimate it immediately follows that
\begin{align*}
\lim_{s\uparrow 1} (1-s) \; I_1
& = \lim_{s\uparrow 1} (1-s) \int_{|x-y|<1} G\left(\left|\Re\left((\nabla u(x)-iA(x)u(x)) \frac{x-y}{|x-y|^s}\right)\right|\right)\frac{ dy}{|x-y|^n}.
\end{align*}

Observe now the following. If $z\in\CC^n$, 
\begin{align*}
\int_{|h|<1} G\left(\frac{|\Re(zh)|}{|h|^s}\right)\, \frac{dh}{|h|^n} &= \int_{|h|<1} G\left(|\Re z|\frac{|h_n|}{|h|^s}\right) \, \frac{dh}{|h|^n}  = \int_0^1 \int_{\Sn} G(|\Re z| |w_n| r^{1-s}) dS_w \frac{dr}{r}.
\end{align*}

Therefore, in view of definition \eqref{phitilde}, we get
\begin{equation} \label{des.I}
\lim_{s\uparrow 1} (1-s)I_1 = \tilde G (|\Re\left(\nabla u(x)-iA(x)u(x)\right)|).
\end{equation}

Finally, since $G$ is increasing and \eqref{cotas} holds, $I_2$ is bounded as
\begin{align} \label{cota.I2}
I_2 &\leq  \int_{|x-y|\geq 1}  \frac{G (2\|u\|_\infty)}{|x-y|^{n+s}} \, dy = G (2\|u\|_\infty) n\omega_n \int_1^\infty  \frac{1}{r^{1+s}}\,dr = \frac{n\omega_n}{s}G (2\|u\|_\infty),
\end{align}
from where we can derive that
\begin{equation} \label{des.II}
\lim_{s\uparrow 1} (1-s)I_2 = 0.
\end{equation}

Summing up, from \eqref{des.I} and \eqref{des.II} we obtain \eqref{eq.bbm}.
\end{proof}

\begin{proof} [Proof of Theorem \ref{main1}]
Given $u\in C_c^2(\R^n;\CC)$ with $\supp(u)\subset B_R(0)$, in view of Lemma \ref{lemma.uno} it only remains to show the existence of an integrable majorant  for
$$
(1-s)F_s^\Re(x):=(1-s)\int_{\R^n} G(|\Re D_s^A u(x,y)|)\, \frac{dy}{|x-y|^n}
$$
and for
$$
(1-s)F_s^\Im(x):=(1-s)\int_{\R^n} G(|\Im D_s^A u(x,y)|)\, \frac{dy}{|x-y|^n}.
$$

We perform all our computations for $F_s^\Re$, since the ones for $F_s^\Im$ are completely analogous. Without loss of generality we can assume that $R>1$.

First, we analyze the behavior of $F_s^\Re(y)$ for small values of $y$. When $|y|<2R$ we can write split the integral $F_s^\Re(x)$ as $I_1+I_2$, where the first term corresponds to integrate over $B_1(x)$ and the second one over its complement.

Arguing as in \eqref{cota.I1} and \eqref{cota.I2} we obtain that
\begin{align} \label{dest1}
I_1 \leq \C \frac{n\omega_n}{1-s} \left[ G(\|\nabla u\|_\infty + \|A\|_\infty \|u\|_\infty) + \|\nabla A\|_\infty G(\|u\|_\infty) \right]
\end{align}
and
\begin{align} \label{dest2}
I_2&\leq  \frac{n\omega_n}{s}G(2\|u\|_\infty).
\end{align}

When $|x|\geq 2R$ the function $u$ vanishes and we have that
$$
F_s^\Re(x)= \int_{B_R(0)} G\left( \frac{|\Re(e^{i(x-y)A(\frac{x+y}{2})}u(y)|}{|x-y|^s} \right)\, \frac{dy}{|x-y|^n}.
$$
Since $|x-y| \geq |x|-R\geq \frac12|x|$, from the monotonicity of $G$, \eqref{H2} and \eqref{P5} (since $|x|\geq 2$) we get 
\begin{align} \label{dest3}
\begin{split}
|F_s^\Re(x)|&\leq \frac{2^n}{|x|^n}\int_{B_R(0)} G\left( \frac{2^s |u(y)| }{|x|^s} \right) \, dy \leq \frac{C}{|x|^n}\int_{B_R(0)} G\left( \frac{ |u(y)| }{|x|^s} \right) \, dy\\
&\leq \frac{\C}{|x|^{n+s}}\int_{B_R(0)} G(|u(y)|) \,dy \leq \frac{\C}{|x|^{n+\frac12 }}\int_{B_R(0)} G(|u(y)|) \, dy,
\end{split}	
\end{align}
for any $s\ge \frac12$.

From \eqref{dest1}, \eqref{dest2} and \eqref{dest3}, there is $C=C(n,G,u)$ independent of $s$ such that
$$
(1-s)|F_s^\Re(x)|\leq   C\left(\chi_{|x|\leq R}(x) + \frac{1}{|x|^{n+\frac12}} \chi_{|x|\geq R}(x) \right) \in L^1(\R^n).
$$

Then, from Lemma \ref{lemma.uno} and the Dominated Convergence Theorem the result follows for any $u\in C_c^2(\R^n;\CC)$.

Let us extend the result for any $u\in W^{1,G}_A(\R^n)$. According to Proposition \ref{es.denso}, let $\{u_k\}_{k\in\N}\subset C_c^2(\R^n;\CC)$ be a sequence such that $u_k\to u$ in $W^{1,G}_A(\R^n)$. Then
\begin{align} \label{desig00}
\begin{split}
\left| (1-s)I_{s,G}^A(u) - I_{\tilde G}^A(u) \right| \leq &  (1-s) \left| I_{s,G}^A(u) - I_{s,G}^A(u_k) \right| + \left|(1-s) I_{s,G}^A(u_k) - I_{\tilde G}^A (u_k) \right|\\
& + \left|I_{\tilde G}^A(u_k) - I_{\tilde G}^A(u)\right|.
\end{split}
\end{align}
Let us fix $\ve>0$. Since the modular $I_{\tilde G}^A$ is continuous on $W^{1,G}_A(\R^n)$ and since $u_k\to u$ in $W^{1,G}_A(\R^n)$, it follows that there exists $k_0$ such that for $k\geq k_0$, 
$$
	|I_{\tilde G}^A(  u_k )-I_{\tilde G}^A(u)|\leq \frac{\ve}{2},
$$
and using \cite[Lemma 2.6]{FBS} one can take $\delta>0$ (to be fixed) such that
\begin{equation} \label{desig1}
(1-s) | I_{s,G}^A(u) -  I_{s,G}^A(u_k) | \leq (1-s)\delta  I_{s,G}^A(u_k)+ (1-s)C_\delta I_{s,G}^A(u-u_k).
\end{equation}
Observe that from Lemma \ref{teo1} we have that $(1-s)  I_{s,G}^A(u_k) \leq K$ for some positive constant $K$. Moreover, again from Lemma \ref{teo1}, there is some $k_1$ such that for $k\geq k_1$ it holds that $
(1-s)I_{s,G}^A(u-u_k) \leq \frac{\ve}{4C_\delta}$. Consequently, it follows that \eqref{desig1} can be bounded as
$$
(1-s) | I_{s,G}^A(u) -  I_{s,G}^A(u) |   \leq \delta K + \frac{\ve}{4}
$$
for $k\geq k_1$. Hence, choosing $\delta=\frac{\ve}{4K}$ we find that \eqref{desig00} is upper bounded as
$$
\left| (1-s)I_{s,G}^A(u) - I_{1,\tilde G}^A(u) \right| \leq \ve + \left|(1-s) I_{s,G}^A(u_k) - I_{1,\tilde G}^A (u_k) \right|
$$
for all $k\geq \max\{k_0,k_1\}$. Finally, the desired result follows by fixing a value of $k\geq \max\{k_0,k_1\}$ and taking limit as $s\uparrow 1$.

To finish the proof, let us see that if $u\in L^G(\R^n;\CC)$ is such that
$$
\liminf_{s\uparrow 1} (1-s) I_{s,G}^A(u)<\infty,
$$
then $u\in W^{1,G}_A(\R^n)$.

Given $u\in L^G(\R^n;\CC)$, according to Lemmas \ref{lema.reg} and \ref{lema.trunc}, if we define the approximating family 
$$
	u_{k,\ve}=\rho_\ve * (u\eta_k) \in C_c^\infty(\R^n;\CC),
$$
it satisfies
$$
\liminf_{s\uparrow 1} (1-s) I_{s,G}(u_{k,\ve})<C,
$$
with $C$ independent on $\ve>0$ and $k\in\N$. 

The first part of this theorem gives that
$$
I_{1,\tilde G}^A( u_{k,\ve}) =\lim_{s\uparrow 1} (1-s) I_{s,G}^A(u_{k,\ve})<C,
$$
then, from Remark \ref{rema}, $\{u_{k,\ve}\}_{k\in\N, \ve>0}$ is bounded in $W^{1,G}_A(\R^n)$. Consequently, from Proposition \ref{propiedades1}, there exists a sequence $u_j=u_{k_j,\ve_j}$ with $k_j\to\infty$ and $\ve_k\downarrow 0$ and $\tilde u\in W^{1,G}_A(\R^n)$ such that $u_j\cd \tilde u$ weakly in $W^{1,G}_A(\R^n)$.  Moreover, since $u_{k,\ve} \to u$ in $L^G(\R^n;\CC)$ as $k\to\infty$, and $\ve\downarrow 0$, we can conclude that $\tilde u= u\in W^{1,G}_A(\R^n)$ as required.
\end{proof}

\section{Some consequences and applications}

In this final section, we show some immediate consequences of Theorem \ref{main1}. This section can be seen as a follow up of \cite[Section 6]{FBS} where the same type of applications were derived for the case of $A\equiv 0$.

Throughout this section $G$ will be an Orlicz function satisfying \eqref{lieberman} such that the limit in \eqref{phitilde} exists.

When working on a domain $\Omega\subset \R^n$ (bounded or not) it is useful to introduce the following notations.

The space $W^{1,G}_{A,0}(\Omega)$ denotes, as usual, is defined as the closure of $C^\infty_c(\Omega;\CC)$ with respect to the $\|\cdot\|_{1,A,G}-$norm.

In the fractional setting, we use the following definitions
$$
W^{s,G}_{A,0}(\Omega) := \{u\in W^{s,G}_A(\R^n) \colon u=0 \text{ a.e. in } \R^n\setminus \Omega \}.
$$

Alternatively, one can consider
$$
\widetilde{W}^{s,G}_A(\Omega) := \overline{C_c^\infty(\Omega)}^{\|\cdot\|_{s,A,G}}.
$$

In the classical case, i.e. when $G(t)=t^p$ and $A=0$, these spaces $W^{s,p}_0(\Omega)$ and $\widetilde{W}^{s,p}(\Omega)$ are known to coincide when $s<\tfrac{1}{p}$ or when $0<s<1$ and $\Omega$ has Lipschitz boundary. See \cite{DPV}.

In this paper, we shall not investigate the cases where these spaces $W^{s,G}_{A,0}(\Omega)$ and $\widetilde{W}^{s,G}_A(\Omega)$ coincide and use the space $W^{s,G}_{A,0}(\Omega)$ to illustrate our applications.

In what follows, every function $u\in L^G(\Omega;\CC)$ it will be assumed to be extended by 0 to $\R^n\setminus \Omega$.

Finally, observe that the inclusions
$$
W^{s,G}_{A,0}(\Omega)\subset W^{s,G}_A(\R^n)\subset L^G(\R^n;\CC)
$$
imply
$$
L^{G^*}(\Omega;\CC)\subset L^{G^*}(\R^n;\CC)\subset W^{-s,G^*}_A(\R^n)\subset W^{-s,G^*}_A(\Omega),
$$
where $W^{-s,G^*}_A(\Omega)$ denotes the (topological) dual space of $W^{s,G}_{A,0}(\Omega)$.

\subsection{Poincar\'e's  inequality}

A first consequence that we get is the Poincar\'e's inequality.

Poincar\'e's inequality in the magnetic setting is a straightforward consequence of the so-called {\em diamagnetic inequality}. This inequality for the classical setting is well-known (see for instance \cite[Theorem 7.21]{LossLieb})
\begin{thm}
Let $A\colon \Omega\to\R^n$ be a measurable magnetic potential such that $|A|<\infty$ a.e. in $\Omega$ and let $u\in W^{1,1}_\text{loc}(\R^n;\CC)$. Then the following {\em diamagnetic inequality} holds
\begin{equation}\label{DI}
|\nabla |u|(x)| \le |\nabla u(x) - iA(x) u(x)|,
\end{equation}
for a.e. $x\in\Omega$.
\end{thm}

The fractional analog of \eqref{DI} was provided in \cite[Lemma 3.1 and Remark 3.2]{AvSq}, namely:
\begin{thm}
Let $A\colon \R^n\to\R^n$ be a measurable magnetic potential such that $|A|<\infty$ a.e. in $\R^n$ and let $u\colon \R^n\to \CC$ be a measurable function such that $|u|<\infty$ a.e. in $\R^n$. Then, the following {\em fractional diamagnetic inequality} holds
\begin{equation}\label{FDI}
||u(x)|-|u(y)||\le |e^{-i(x-y)A(\tfrac{x+y}{2})} u(x) - u(y)|,
\end{equation}
for a.e. $x,y\in\R^n$.
\end{thm}

\begin{rem}
Observe that the fractional diamagnetic inequality \eqref{FDI} can be stated as
\begin{equation}\label{FDI2}
|D_s|u|(x,y)|\le |D_s^A u(x,y)|,
\end{equation}
a.e. $x, y\in \R^n$, where $D_s v(x,y) = \frac{v(x)-v(y)}{|x-y|^s}$.
\end{rem}

With the help of these diamagnetic inequalities \eqref{DI} and \eqref{FDI} it is easy to prove  a Poincar\'e inequality in the context of Orlicz-Sobolev and fractional Orlicz-Sobolev spaces.

First recall the classical Poincar\'e inequality in Orlicz-Sobolev spaces. Even though it is well known, we include a proof here for the reader convenience and to recall a precise estimate of the constant.
\begin{thm}\label{poincare.1}
Let $\Omega\subset\R^n$ be a bounded domain and $G\colon \R\to\R$ be an Orlicz function. Then, for every $u\in W^{1,G}_0(\Omega)$,
$$
I_G(u)\le I_G(\dd |\nabla u|),
$$
where $\dd=\diam(\Omega)$.
\end{thm}

\begin{proof}
The proof is standard. Let assume first that $u\in C^\infty_c(\Omega)$, $x_0\in\partial\Omega$ be fixed and for any $x\in\Omega$ we get the estimate
$$
u(x) = u(x) - u(x_0) = \int_0^1 \frac{d}{dt} u(x_0 + t(x-x_0))\, dt \le |x-x_0|\int_0^1 |\nabla u(x_0+t(x-x_0))|\, dt.
$$
Now we use that $|x-x_0|\le \dd$ and Jensen's inequality to obtain
$$
G(|u(x)|)\le \int_0^1 G(\dd |\nabla u(x_0+t(x-x_0))|)\, dt.
$$
Finally, we integrate in $\R^n$ with respect to $x$ and apply Fubini's theorem to conclude the desired result.

The proof for general $u\in W^{1,G}_0(\Omega)$ follows by a density argument.
\end{proof}

The Poincar\'e inequality for fractional order Orlicz-Sobolev spaces was proved in \cite[Theorem 2.12]{Bonder-Perez-Salort}.

\begin{thm}\label{poincare.s}
Let $\Omega\subset\R^n$ be a bounded domain and $G\colon \R\to\R$ be an Orlicz function satisfying \eqref{lieberman}. Then, for every $0<s<1$ and every $u\in W^{s,G}_0(\Omega)$,
$$
I_G(|u|)\le I_{s,G}((1-s) C \dd^s u),
$$
where $\dd=\diam(\Omega)$ and $C$ depends on $n, p_+$ and $p_-$.
\end{thm}

Combining the Poincar\'e's inequalities of Theorems \ref{poincare.1} and \ref{poincare.s} together with the diamagnetic inequalities \eqref{DI} and \eqref{FDI} we can easily prove the Poincar\'e inequalities for the Magnetic Orlicz-Sobolev and fractional Orlicz-Sobolev spaces.
\begin{thm}
Let $\Omega\subset\R^n$ be a bounded domain, $G\colon \R\to\R$ be an Orlicz function satisfying \eqref{lieberman} and $0<s<1$. Then, there exists a constant $C=C(n,p_-, p_+)$ such that
$$
I_G(u) \le I_{s,G}^A((1-s)C\dd^s u),
$$
for every $u\in W^{s,G}_{A,0}(\Omega)$, where $\dd=\diam(\Omega)$.

Moreover, for every $u\in W^{1,G}_{A,0}(\Omega)$, it holds
$$
I_G(u)\le I_G^A(\dd u).
$$
\end{thm}

\begin{proof}
First let us deal with the case $s=1$. 

In this case we use Theorem \ref{poincare.1} and \eqref{DI} to conclude that
$$
I_G(u)\le I_G(\dd |\nabla |u||) \le  I_G(\dd |\nabla u -iAu|) = I_G^A(\dd u).
$$

Now, for the case $0<s<1$, we use Theorem \ref{poincare.s} and \eqref{FDI2} to conclude that
$$
I_G(u)\le I_{s, G} ((1-s)C \dd^s |u|)\le I_s^A((1-s)C \dd^s u).
$$
This finishes the proof.
\end{proof}

As a simple corollary we obtain the Poincar\'e inequality for Luxemburg norms.
\begin{cor}
Under the previous assumptions, there exist a constant $C=C(n,p_-,p_+,\dd)$ such that
$$
\|u\|_G\le (1-s)C|u|_{s, G}^A,
$$
for every $u\in W^{s, G}_{A,0}(\Omega)$, $0<s\le 1$.
\end{cor}

\subsection{$\Gamma-$convergence}

Let us recall the definition of $\Gamma-$convergence.

\begin{defn}
Let $(X,d)$ be a metric space and $F,F_k\colon X  \to \bar \R$. We say that $F_k$  $\Gamma-$converges to $F$ if for every $x\in X$ the following conditions are valid.

\begin{itemize}
\item[(i)] (lim inf inequality) For every sequence $\{x_k\}_{k\in\N}\subset X$ such that $x_k \to x$ in $X$, 
$$
F(x)\leq \liminf_{k\to\infty} F_k(x_k).
$$

\item[(ii)] (lim sup inequality). For every $x\in X$, there is a sequence $\{y_k\}_{k\in\N}\subset X$ converging to $x$ such that 
$$
F(x)\geq  \limsup_{k\to\infty} F_k(y_k).
$$
This sequence $\{y_k\}_{k\in\N}$ is usually called as the {\em recovery sequence}.
\end{itemize}

The functional $F$ is called the $\Gamma-$limit of the sequence $\{F_k\}_{k\in\N}$ and it is denoted by $F_k \stackrel{\Gamma}{\to} F$ and 
$$
F=\glim_{k\to\infty} F_k.
$$
\end{defn}

\begin{rem}
In the case where the functions are indexed by a continuous parameter, $\{F_\ve\}_{\ve>0}$, we say that 
$$
F=\glim_{\ve\downarrow 0} F_\ve,
$$
if and only if for every sequence $\ve_k\downarrow 0$, it follows that $F_{\ve_k}\stackrel{\Gamma}{\to} F$.
\end{rem}

\medskip

Now, let us fix $\Omega\subset \R^n$ open, and an Orlicz function $G$.

For any $0<s<1$, we define the functional $\J_s\colon L^G(\Omega;\CC)\to \bar \R$ by
\begin{align*}
\J_s(w)=\begin{cases}
(1-s)I_{s,G}^A(w) &\qquad \text{ if } w\in W_{A,0}^{s,G}(\Omega)\\
+\infty &\qquad \text{ otherwise},
\end{cases}
\end{align*}
and the limit functional $\J\colon L^G(\Omega;\CC)\to \bar \R$ 
\begin{align*}
\J(w)=\begin{cases}
I_{1, \tilde G}^A(w) &\text{ if } w\in W_{A, 0}^{1,\tilde G}(\Omega) \\
+\infty &\text{ otherwise}.
\end{cases}
\end{align*}

\begin{thm} \label{teo.gamma.conv}
With the previous notation we have that
$$
\J= \glim_{\ve\downarrow 0} \J_{1-\ve}.
$$
\end{thm}
The proof of Theorem \ref{teo.gamma.conv} is a direct consequence of our previous results. Indeed, the limsup inequality follows just by choosing the constant sequence as the recovery sequence, whilst the liminf is is the content of the next proposition.
\begin{prop}\label{limsup.prop}
Let $G$ be an Orlicz function such that the limit in \eqref{phitilde} exists. Let $\{u_\ve\}_{\ve>0}\subset L^G(\Omega;\CC)$ such that $u_\ve\to u$ in $L^G(\Omega;\CC)$. Then
$$
J(u)\le \liminf_{\ve\to 0} J_\ve(u_\ve).
$$
\end{prop}

\begin{proof}
Let $\ve_k\downarrow 0$ and denote $u_k:=u_{\ve_k}$. Since $u_k\to u$ in $L^G(\Omega;\CC)$, we can assume that $u_k\to u$ a.e. in $\R^n$. 

We can also assume, without loss of generality, that $\sup_k J_{\ve_k}(u_k) <\infty$ and therefore, by Lemma \ref{lema.A.0} and \cite[Theorem 5.1]{FBS}, we obtain that $u\in W^{s,\tilde G}_A(\R^n)$.

Therefore, we can apply Theorem \ref{main1} to the function $u$ to conclude that, for any $\delta>0$, there exists $s_\delta\in (0,1)$ such that
\begin{equation}\label{putamierda}
(1-\delta) I_{\tilde G}^A(u)\le (1-s) I_{s,G}^A(u),
\end{equation}
for every $s\in (s_\delta, 1)$.

Observe that by Fatou's lemma we have that, for any $s\in (0,1)$
\begin{equation}\label{putamierda2}
(1-s)I_{s, G}^A(u)\le \liminf_{k\to\infty} (1-s) I_{s, G}^A(u_k).
\end{equation}

Combining \eqref{putamierda} and \eqref{putamierda2}, we obtain the existence of $k_\delta$ such that
\begin{equation}\label{putamierda3}
(1-\delta) I_{\tilde G}^A(u)\le (1+\delta) (1-s) I_{s, G}^A(u_k),
\end{equation}
for every $k\ge k_\delta$ and every $s_\delta<s<1$. So from \eqref{putamierda3} we conclude that
$$
\frac{1-\delta}{1+\delta} I_{\tilde G}^A(u)\le \liminf_{k\to\infty} (1-s_k) I_{s_k, G}^A(u_k) = \liminf_{k\to\infty} J_{\ve_k}(u_k).
$$

Now the result follows taking $\delta\downarrow 0$.
\end{proof}

The main feature of the $\Gamma-$convergence is that it implies the convergence of minima. 
\begin{thm} \label{teo.Gamma}
Let $(X,d)$ be a metric space and let $F, F_k\colon X  \to \bar \R$, $k\in\N$, be such that $F_k$  $\Gamma-$converges to $F$. Assume that for each $k\in\N$ there exist $x_k\in X$ such that $F_k(x_k)=\inf_{X} F_k$ and suppose that the sequence $\{x_k\}_{k\in\N}\subset X$ is precompact.

Then every accumulation point of $\{x_k\}_{k\in\N}$ is a minimum of $F$ and
$$
\inf_{X} F = \lim_{k\to\infty} \inf_{X} F_k.
$$
\end{thm} 

The proof of Theorem \ref{teo.Gamma} is elementary. For a comprehensive study of $\Gamma-$convergence and its properties, see \cite{DalMaso}.

Consider now $f\in L^{G^*}(\Omega;\CC)$ and define the functionals $\F, \F_\ve$ as 
\begin{equation}\label{FjF}
\F_\ve(u) := \J_{1-\ve}(u) - \int_\Omega \Re(f \bar u)\, dx\quad \text{and}\quad \F(w) := \J(u) - \int_\Omega \Re(f \bar u)\, dx.
\end{equation}

Since $u\mapsto \int_\Omega \Re(f \bar u)\,dx$ is continuous in $L^G(\Omega;\CC)$, Theorem \ref{teo.gamma.conv} implies that $\F_\ve\stackrel{\Gamma}{\to} \F$. See \cite[Proposition 6.21]{DalMaso}.

Let us apply Theorem \ref{teo.Gamma} to the family $\F_\ve$. With this aim, let us verify that, given $\ve_k\downarrow 0$, there exists a sequence $\{u_k\}_{k\in\N}\in L^G(\Omega;\CC)$ of minimizers of $\F_{\ve_k}$ which is precompact in $L^G(\Omega;\CC)$.

The proof of the next lemma is standard. We state it for future references and leave the proof to the reader.
\begin{lema}\label{lema1}
Let $\ve>0$, $G$ be a uniformly convex Orlicz function and $f\in L^{G^*}(\Omega;\CC)$. Then there exists a unique function $u\in W^{s,G}_{A,0}(\Omega)$ such that
$$
\F_\ve(u)= \inf_{v\in W^{s,G}_{A,0}(\Omega)} \F_\ve(v) = \min_{v\in W^{s,G}_{A,0}(\Omega)} \F_\ve(v).
$$
\end{lema}

Now, a simple consequence of Lemma \ref{lema.A.0} and \cite[Theorem 5.1]{FBS} gives the compactness of the sequence of minima. Again, the details of the proof are left to the readers.
\begin{lema}\label{lema2}
Let $\ve_k\downarrow 0$, and $\Omega\subset \R^n$ be an open bounded subset. Given $k\in\N$, let $u_k\in L^G(\Omega;\CC)$ be the minimum of $\F_{\ve_k}$. Then  $\{u_k\}_{k\in\N}\subset L^G(\Omega;\CC)$ is precompact.
\end{lema}

As a corollary of Lemmas \ref{lema1} and \ref{lema2} and Theorem \ref{teo.Gamma} we obtain the following result.
\begin{thm}\label{teo.conv.minimos}
Let $G$ be a uniformly convex Orlicz function, $\Omega\subset\R^n$ be open and bounded and let $u_\ve\in  L^G(\Omega;\CC)$ be the minimum of $\F_\ve$. Then there exists $u\in L^G(\Omega;\CC)$ such that 
$$
u=\lim_{\ve\downarrow 0} u_\ve \text{ in } L^G(\Omega;\CC) \qquad \text{ and } \qquad \F(u)=\min_{v\in L^G(\Omega;\CC)} \F(v).
$$
\end{thm}

Finally, Theorem \ref{main3} is a trivial consequence of Theorem \ref{teo.conv.minimos}.

\section*{Acknowledgements}

This paper is partially supported by grants UBACyT 20020130100283BA, CONICET PIP 11220150100032CO and ANPCyT PICT 2012-0153.

All of the authors are members of CONICET.

Part of this paper was written while the first author was visiting the University of Nottingham at Ningbo, China (UNNC). He want to thank the UNNC for the kind hospitality.

\bibliographystyle{amsplain}
\bibliography{biblio}

\providecommand{\bysame}{\leavevmode\hbox to3em{\hrulefill}\thinspace}
\providecommand{\MR}{\relax\ifhmode\unskip\space\fi MR }
\providecommand{\MRhref}[2]{%
  \href{http://www.ams.org/mathscinet-getitem?mr=#1}{#2}
}
\providecommand{\href}[2]{#2}
\begin{thebibliography}{10}

\bibitem{BGKSM}
Michel {Bonnefont}, Sylvain {Gol{\`e}nia}, Matthias {Keller}, Shiping {Liu},
  and Florentin {M{\"u}nch}, \emph{Magnetic sparseness and schr\" odinger
  operators on graphs}, ArXiv e-prints (2017), arXiv:1711.10418.

\bibitem{BBM}
Jean Bourgain, Haim Brezis, and Petru Mironescu, \emph{Another look at sobolev
  spaces}, in Optimal Control and Partial Differential Equations, 2001,
  pp.~439--455.

\bibitem{DalMaso}
Gianni Dal~Maso, \emph{An introduction to {$\Gamma$}-convergence}, Progress in
  Nonlinear Differential Equations and their Applications, vol.~8, Birkh\"auser
  Boston, Inc., Boston, MA, 1993. \MR{1201152}

\bibitem{AvSq}
Pietro d'Avenia and Marco Squassina, \emph{Ground states for fractional
  magnetic operators}, ESAIM Control Optim. Calc. Var. \textbf{24} (2018),
  no.~1, 1--24. \MR{3764131}

\bibitem{DPV}
Eleonora Di~Nezza, Giampiero Palatucci, and Enrico Valdinoci,
  \emph{Hitchhiker's guide to the fractional {S}obolev spaces}, Bull. Sci.
  Math. \textbf{136} (2012), no.~5, 521--573. \MR{2944369}

\bibitem{Bonder-Perez-Salort}
J.~{Fern{\'a}ndez Bonder}, M.~{P{\'e}rez-Llanos}, and A.~M. {Salort}, \emph{{A
  H\"older infinity Laplacian obtained as a limit of Orlicz fractional
  Laplacians}}, ArXiv e-prints (2018).

\bibitem{FBS}
J.~{Fern{\'a}ndez Bonder} and A.~M. {Salort}, \emph{{Fractional order
  Orlicz-Sobolev spaces}}, ArXiv e-prints (2017).

\bibitem{GSW}
Piotr Gwiazda, Agnieszka \'{S}wierczewska Gwiazda, and Aneta Wr\'{o}blewska,
  \emph{Monotonicity methods in generalized {O}rlicz spaces for a class of
  non-{N}ewtonian fluids}, Math. Methods Appl. Sci. \textbf{33} (2010), no.~2,
  125--137. \MR{2597212}

\bibitem{Ic1}
Takashi Ichinose, \emph{Essential selfadjointness of the {W}eyl quantized
  relativistic {H}amiltonian}, Ann. Inst. H. Poincar\'{e} Phys. Th\'{e}or.
  \textbf{51} (1989), no.~3, 265--297. \MR{1034589}

\bibitem{Ic2}
\bysame, \emph{Magnetic relativistic {S}chr\"{o}dinger operators and
  imaginary-time path integrals}, Mathematical physics, spectral theory and
  stochastic analysis, Oper. Theory Adv. Appl., vol. 232,
  Birkh\"{a}user/Springer Basel AG, Basel, 2013, pp.~247--297. \MR{3077280}

\bibitem{Ic3}
Takashi Ichinose and Hiroshi Tamura, \emph{Imaginary-time path integral for a
  relativistic spinless particle in an electromagnetic field}, Comm. Math.
  Phys. \textbf{105} (1986), no.~2, 239--257. \MR{849207}

\bibitem{KrRu61}
M.~A. Krasnosel'ski\u{\i} and Ja.~B. Ruticki\u{\i}, \emph{Convex functions and
  {O}rlicz spaces}, Translated from the first Russian edition by Leo F. Boron,
  P. Noordhoff Ltd., Groningen, 1961. \MR{0126722}

\bibitem{LZ}
Sihua Liang and Jihui Zhang, \emph{On some {$p$}-{L}aplacian equation with
  electromagnetic fields and critical nonlinearity in {${\mathbb R}^N$}}, J.
  Math. Phys. \textbf{56} (2015), no.~4, 041504, 11. \MR{3390938}

\bibitem{LossLieb}
Elliott~H. Lieb and Michael Loss, \emph{Analysis}, second ed., Graduate Studies
  in Mathematics, vol.~14, American Mathematical Society, Providence, RI, 2001.
  \MR{1817225}

\bibitem{NPSV}
Hoai-Minh Nguyen, Andrea Pinamonti, Marco Squassina, and Eugenio Vecchi,
  \emph{New characterizations of magnetic {S}obolev spaces}, Adv. Nonlinear
  Anal. \textbf{7} (2018), no.~2, 227--245. \MR{3794886}

\bibitem{SquassinaVolzone}
Marco Squassina and Bruno Volzone, \emph{Bourgain-{B}r\'ezis-{M}ironescu
  formula for magnetic operators}, C. R. Math. Acad. Sci. Paris \textbf{354}
  (2016), no.~8, 825--831. \MR{3528339}

\bibitem{WK}
Aneta Wr\'{o}blewska-Kami\'{n}ska, \emph{Unsteady flows of non-{N}ewtonian
  fluids in generalized {O}rlicz spaces}, Discrete Contin. Dyn. Syst.
  \textbf{33} (2013), no.~6, 2565--2592. \MR{3007700}

\end{thebibliography}

\end{document}